\documentclass[12pt]{amsart}
\usepackage{amssymb,amsmath,amsthm,mathrsfs,multirow,xcolor,framed,url}
\usepackage[pdftex,
         pdfauthor={Dandan Chen, Rong Chen and Frank Garvan},
         pdftitle={Congruences modulo powers of 5 for the rank parity function},
         pdfsubject={MATHEMATICS},
         pdfkeywords={partition congruences, Dyson's rank, mock theta functions, modular functions},
         pdfproducer={Latex with hyperref},
         pdfcreator={pdflatex}]{hyperref}
\newtheorem{theorem}{Theorem}[section]
\newtheorem{lemma}[theorem]{Lemma}
\newtheorem{cor}[theorem]{Corollary}

\theoremstyle{definition}

\theoremstyle{remark}
\newtheorem{remark}[theorem]{Remark}

\numberwithin{equation}{section}

\newcommand\nutwid{\overset {\text{\lower 3pt\hbox{$\sim$}}}\nu}
\newcommand\newaa{\alpha}                                               
\newcommand{\abs}[1]{\lvert#1\rvert}

\newcommand{\SLZ}{\mbox{SL}_2(\mathbb{Z})}       
\newcommand{\uhp}{\mathscr{H}}  

\newcommand\FL[1]{\left\lfloor#1\right\rfloor}
\newcommand\CL[1]{\left\lceil#1\right\rceil}
\newcommand\stroke[2]{{#1}\,\left\arrowvert\,#2\right.}

\newcommand\abcdMAT{\begin{pmatrix} a & b \\ c & d \end{pmatrix}}
\newcommand\GoneMAT{\begin{pmatrix} 1 & * \\ 0 & 1 \end{pmatrix}}
\newcommand\TMAT{\begin{pmatrix} 1 & 1 \\ 0 & 1 \end{pmatrix}}
\newcommand\MAT[4]{\begin{pmatrix} {#1} & {#2} \\ {#3}  & {#4} \end{pmatrix}}
\newcommand\Z{\mathbb{Z}}
\newcommand\C{\mathbb{C}}
\allowdisplaybreaks

\newcommand\omylabel[1]{\label{#1}}

\newcommand\omyeqn[1]{(\ref{eq:#1})}

\newcommand\omycite[1]{}

\newcommand\omylem[1]{\ref{lem:#1}}

\newcommand\omythm[1]{\ref{thm:#1}}
\newcommand\thm[1]{\ref{thm:#1}}
\newcommand\lem[1]{\ref{lem:#1}}
\newcommand\corol[1]{\ref{cor:#1}}

\newcommand\eqn[1]{(\ref{eq:#1})}
\newcommand\sect[1]{\ref{sec:#1}}

\newcommand\subsect[1]{\ref{subsec:#1}}

\newcommand{\beqs}{\begin{equation*}}
\newcommand{\eeqs}{\end{equation*}}
\newcommand{\beq}{\begin{equation}}
\newcommand{\eeq}{\end{equation}}
\renewcommand{\MR}[1]{\href{http://www.ams.org/mathscinet-getitem?mr={#1}}{MR{#1}}}
\DeclareMathOperator{\IM}{Im}
\DeclareMathOperator{\ORD}{Ord}
\DeclareMathOperator{\ord}{ord}
\newcommand\umin[1]{\underset{#1}{\min}}
\newcommand\submin[1]{\underset{\substack{#1}}{\min}}

\begin{document}
\title[Rank parity function congruences]{Congruences modulo powers of $5$ \\
for the rank parity function}

\author{Dandan Chen}
\address{School of Mathematical Sciences, East China Normal University, 
Shanghai, People's Republic of China}
\email{ddchen@stu.ecnu.edu.cn}
\author{Rong Chen}
\address{School of Mathematical Sciences, East China Normal University, 
Shanghai, People's Republic of China}
\email{rchen@stu.ecnu.edu.cn}
\author{Frank Garvan}
\address{Department of Mathematics, University of Florida, Gainesville,
FL 32611-8105}
\email{fgarvan@ufl.edu}
\thanks{The first and second authors were supported in part by the 
National Natural
Science Foundation of China (Grant No. 11971173) and an ECNU
Short-term Overseas Research Scholarship for Graduate Students
(Grant no. 201811280047).
The third author was supported in part by a grant from 
        the Simon's Foundation (\#318714).}


\date{December 24, 2020}                   



\begin{abstract}
It is well known that Ramanujan conjectured congruences modulo powers of
$5$, $7$ and and $11$ for the partition function. These were subsequently
proved by Watson (1938) and Atkin (1967). In 2009 Choi, Kang, and
Lovejoy proved congruences modulo powers of $5$ for the crank parity
function. The generating function for rank parity function is $f(q)$,
which is the first example of a mock theta function that Ramanujan mentioned
in his last letter to Hardy.
We prove congruences modulo powers of $5$ for 
the rank parity function.
\end{abstract}

\maketitle

\section{Introduction}
\omylabel{sec:intro}
Let $p(n)$ be the number of unrestricted partitions of $n$. 
Ramanujan discovered and later proved that
\begin{align}
p(5n+4) &\equiv 0 \pmod{5},\omylabel{eq:ram5} \\ 
p(7n+5) &\equiv 0 \pmod{7},\omylabel{eq:ram7} \\ 
p(11n+6) &\equiv 0 \pmod{11}.\omylabel{eq:ram11} 
\end{align}
In 1944 Dyson \cite{Dy44}\omycite{Dy44} 
defined the \textit{rank} of a partition as 
the largest part minus the number of parts and conjectured
that residue of the rank mod $5$ (resp. mod $7$) divides the partitions
of $5n+4$ (resp. $7n+5$) into $5$ (resp. $7$) equal classes thus giving
combinatorial explanations of Ramanujan's partition congruences mod $5$
and $7$. Dyson's rank conjectures were proved by 
Atkin and Swinnerton-Dyer \cite{At-SwD}\omycite{At-SwD}.
Dyson also conjectured the existence of another statistic he called the
\textit{crank} which would likewise explain Ramanujan's partition
congruence mod $11$. The crank was found by Andrews and the third author 
\cite{An-Ga88}\omycite{An-Ga88}
who defined the \textit{crank} as the largest part, if the partition has no 
ones,
and otherwise as the difference between the number of parts larger than the
number of ones and the number of ones.

Let $M_e(n)$ (resp. $M_o(n)$) denote the number of partitions of $n$ with even (resp. odd) crank. Choi, Kang and Lovejoy \cite{Ch-Ka-Lo}\omycite{Ch-Ka-Lo} proved congruences
modulo powers of $5$ for the difference, which we call the
\textit{crank parity function}.

\begin{theorem}[{Choi, Kang and Lovejoy \cite[Theorem 1.1]{Ch-Ka-Lo}\omycite{Ch-Ka-Lo}}]
\omylabel{thm:crankthm}
For all $\newaa\ge0$ we have
$$
M_e(n) - M_o(n) \equiv 0 \pmod{5^{\newaa+1}},\qquad
\mbox{if $24n\equiv 1 \pmod{5^{2\newaa+1}}$}.
$$
\end{theorem}
This gave a weak refinement of Ramanujan's partition congruence modulo
powers of $5$:
$$
p(n) \equiv 0 \pmod{5^a},\qquad
\mbox{if $24n\equiv 1 \pmod{5^\newaa}$}.
$$
This was proved by Watson \cite{Wa38}\omycite{Wa38}.

In this paper we prove an analogue of Theorem \thm{crankthm} for the rank parity function.
Analogous to $M_e(n)$ and $M_o(n)$ we let $N_e(n)$ (resp. $N_o(n)$) denote the number of partitions of $n$ with even 
(resp. odd) rank. It is well known that the difference is related to
Ramanujan's mock theta function $f(q)$. This is the first example of 
a mock theta function that Ramanujan gave in his last letter to 
Hardy. Let
\begin{align*}
  f(q) &= \sum_{n=0}^\infty a_f(n) q^n = 1 + \sum_{n=1}^\infty 
\frac{q^{n^2}}{(1+q)^2(1+q^2)^2 \cdots (1+q^n)^2}\\
       &=
1+q-2\,{q}^{2}+3\,{q}^{3}-3\,{q}^{4}+3\,{q}^{5}-5\,{q}^{6}+7\,{q}^{7}-
6\,{q}^{8}+6\,{q}^{9}-10\,{q}^{10}+12\,{q}^{11}-11\,{q}^{12}+
 \cdots.
\end{align*}
This function has been studied by many authors. Ramanujan conjectured
an asymptotic formula for the coefficients $a_f(n)$. Dragonette 
\cite{Dr52}\omycite{Dr52} improved this result by finding a Rademacher-type asymptotic
expansion for the coefficients. The error term was subsequently improved
by Andrews \cite{An66a}\omycite{An66a}, Bringmann and Ono \cite{Br-On06}\omycite{Br-On06}, and Ahlgren
and Dunn \cite{Ah-Du19}\omycite{Ah-Du19}. We have
$$
a_f(n) = N_e(n) - N_o(n),
$$
for $n\ge0$.

Our main theorem is
\begin{theorem}
\omylabel{thm:mainthm}
For all $\alpha\ge3$ and all $n\ge 0$ we have
\beq
a_f(5^{\alpha}n + \delta_\alpha)
+ a_f(5^{\alpha-2}n + \delta_{\alpha-2})
\equiv 0 \pmod{5^{ \FL{\tfrac{1}{2}\alpha }}},
\omylabel{eq:rmod5}
\eeq
  where $\delta_\alpha$ satisfies $0 < \delta_\alpha < 5^\alpha$ and
$24\delta_\alpha\equiv1\pmod{5^\alpha}$.
\end{theorem}

Below in Section \subsect{genfunc} we show that the generating function for
$$
a_f(5n-1) + a_f(n/5),
$$
is a linear combination of two eta-products. See Theorem \omythm{af5thm}. This enables us to
use the theory of modular functions to obtain congruences. Our presentation and method
is similar to that Paule and Radu \cite{Pa-Ra12}\omycite{Pa-Ra12}, who solved a
difficult conjecture of Sellers \cite{Se1994} for congruences modulo powers of $5$
for Andrews's two-colored generalized Frobenious partitions \cite{An1984mem}.
In Section \sect{modfuncs} we include the necessary background and algorithms from the theory
of modular functions for proving identities. In Section \sect{rankparity5} we apply the
theory of modular functions to prove our main theorem. In Section \sect{further} we conclude
the paper by discussing congruences modulo powers of $7$ for both the rank and crank
parity functions.

\subsection*{Some Remarks and Notation}
\omylabel{subsec:notation}
Throughout this paper we use the standard $q$-notation.
For finite products we use
$$
(z;q)_n=(z)_n=
\begin{cases}
{\displaystyle\prod_{j=0}^{n-1}(1-zq^j)}, & n>0 \\
1,                         & n=0.
\end{cases}
$$
For infinite products we use
$$
(z;q)_\infty=(z)_\infty = \lim_{n\to\infty} (z;q)_n
=\prod_{n=1}^\infty (1-z q^{(n-1)}),
$$
$$
(z_1,z_2,\dots,z_k;q)_\infty = (z_1;q)_\infty (z_2;q)_\infty
\cdots (z_k;q)_\infty,
$$
$$
[z;q]_\infty = (z;q)_\infty (z^{-1}q;q)_\infty=
\prod_{n=1}^\infty (1-z q^{(n-1)})(1-z^{-1}q^{n}),
$$
$$
[z_1,z_2,\dots,z_k;q]_\infty = [z_1;q]_\infty [z_2;q]_\infty
\cdots [z_k;q]_\infty,
$$
for $\abs{q}<1$ and $z$, $z_1$, $z_2$,\dots, $z_k\ne 0$.
For $\theta$-products we use
\beqs
J_{a,b}=(q^a,q^{b-a},q^b;q^b)_\infty,\quad\mbox{and}\quad
J_b=(q^b;q^b)_\infty,
\eeqs            
and as usual 
\beq
\eta(\tau) = \exp(\pi i\tau/12) \prod_{n=1}^\infty (1 - \exp(2\pi in\tau)
= q^{1/24} \prod_{n=1}^\infty (1- q^n),
\omylabel{eq:etadef}
\eeq
where $\IM(\tau)>0$.

Throughout this paper we let $\lfloor x \rfloor$ denote the largest
integer less and or equal to $x$, and let $\lceil x \rceil$
denote the smallest integer greater than or equal to $x$.

     We need some notation for formal Laurent series. 
See the remarks at the end of \cite[Section 1, p.823]{Pa-Ra12}\omycite{Pa-Ra12}.  
Let $R$ be a ring and $q$ be an indeterminant. We let $R((q))$ denote the
formal Laurent series in $q$ with coefficients in $R$. These are
series of the form
$$
f = \sum_{n\in\mathbb{Z}} a_n \, q^n,
$$ 
where $a_n \ne 0$ for at most finitely many $n < 0$. For $f\ne0$
we define the order of $f$ (with respect to $q$) as the smallest 
integer
$N$ such that $a_N\ne0$ and write $N=\ord_q(f)$. We note that if $f$
is a modular function this coincides with $\ord(f,\infty)$.
See equation \omyeqn{ordfz} below for this other notation.
Suppose $t$ and $f\in R((q)$ and the composition $f\circ t$ is well-defined
as a formal Laurent series. This is the case if $\ord_q(t)>0$. 
The $t$-order of
$$
F = f \circ t = \sum_{n\in\mathbb{Z}} a_n \, t^n,
$$
where $t = \sum_{n\in\mathbb{Z}} b_n \, q^n$, is defined to be the
smallest integer $N$ such that $a_N\ne0$ and write $N=\ord_t(F)$.
For example, if
$$
f = {q}^{-1} + 1 + 2\,q + \cdots, \qquad
t = q^2 + 3q^3 + 5q^4 + \cdots, 
$$
then
$$
F = f \circ t 
 = {t}^{-1} + 1 + 2\,t + \cdots, \qquad
 = q^{-2} - 3{q}^{-1} + 5 + \cdots,
$$
so that
$\ord_q(f) = -1$, $\ord_q(t)=2$, $\ord_t(F)=-1$ and $\ord_q(F) = -2$.

\section{Modular Functions}
\omylabel{sec:modfuncs}

In this section we present the needed theory of modular functions
which we use to prove identities. A general reference is Rankin's book
\cite{Ra}\omycite{Ra}.

\subsection{Background theory}
\omylabel{subsec:bthy}
Our presentation is based on \cite[pp.326-329]{Be-RNIII}\omycite{Be-RNIII}.
Let $\uhp = \{\tau\,:\,\IM(\tau)>0\}$ (the complex upper half-plane).
For each $M=\abcdMAT \in M_2^{+}(\mathbb{Z})$, the set of integer
$2\times 2$ matrix with positive determinant, the bilinear
transformation $M(\tau)$ is defined by
$$
M\tau = M(\tau) = \frac{a\tau +b}{c\tau +d}.
$$
The stroke operator is defined by
$$
\left(\stroke{f}{M}\right)(\tau) = f(M\tau),
$$
and satisfies
$$
\stroke{f}{MS} = \stroke{\stroke{f}{M}}{S}.
$$
The modular group $\Gamma(1)$ is defined by
$$
\Gamma(1) = \left\{\abcdMAT\in M_2^{+}(\mathbb{Z})\,:\, ad -bc=1\right\}.
$$
We consider the following subgroups $\Gamma$ of the modular group
with finite index
$$
\Gamma_0(N) =
\left\{\abcdMAT\in\Gamma(1)\,:\, c\equiv 0 \pmod{N}\right\},
$$
$$
\Gamma_1(N) =
\left\{\abcdMAT\in\Gamma(1)\,:\, \abcdMAT\equiv\GoneMAT\pmod{N} \right\},
$$
Such a group $\Gamma$ acts on
$\uhp \cup \mathbb{Q} \cup {\infty}$ by the transformation $V(\tau)$,
for $V\in\Gamma$ which induces an equivalence relation. We
call a set $\mathscr{F}\subseteq \uhp \cup \mathbb{Q} \cup \{\infty\}$
a \textit{fundamental set} for $\Gamma$ if it contains one element of
each equivalence class. The finite set $\mathcal{F} \cap \left(\mathbb{Q}
\cup \{\infty\}\right)$ is called the \textit{complete set of inequivalent cusps}.

A function $f\,:\,\mathscr{H} \longrightarrow \mathbb{C}$ is
a \textit{modular function} on $\Gamma$ if the following conditions
hold:
\begin{enumerate}
\item[(i)] $f$ is holomorphic on $\uhp$.
\item[(ii)] $\displaystyle \stroke{f}{V} = f$ for all $V\in\Gamma$.
\item[(iii)] For every $A\in\Gamma(1)$ the function $\stroke{f}{A^{-1}}$
has an expansion
$$
(\stroke{f}{A^{-1}})(\tau) = \sum_{m=m_0}^\infty b(m) \exp(2\pi i\tau m/\kappa)
$$
on some half-plane $\left\{\tau\,:\,\IM \tau > h \ge 0\right\}$,
where $T=\TMAT$ and
$$
\kappa = \min \left\{k>0\,:\, \pm A^{-1} T^k A \in \Gamma\right\}.
$$
\end{enumerate}
The positive integer $\kappa = \kappa(\Gamma;\zeta)$
is called the \textit{fan width} of
$\Gamma$ at the cusp $\zeta = A^{-1}\infty$.
If $b(m_0)\ne 0$, then we write
$$
\ORD(f,\zeta,\Gamma) = m_0
$$
which is called the \textit{order} of $f$ at $\zeta$ with respect to
$\Gamma$. We also write
\beq
\ord(f;\zeta) = \frac{m_0}{\kappa} = \frac{m_0}{\kappa(\Gamma,\zeta)},
\omylabel{eq:ordfz}
\eeq
which is called the \textit{invariant order} of $f$ at $\zeta$.
For each $z\in\uhp$, $\ord(f;z)$ denotes the order of
$f$ at $z$ as an analytic function of $z$, and the order of $f$ with
respect to $\Gamma$ is defined by
$$
\ORD(f,z,\Gamma) = \frac{1}{\ell} \ord(f;z)
$$
where $\ell$ is the order of $z$ as a fixed point of $\Gamma$.
We note $\ell =1$, $2$ or $3$.
Our main tool for proving modular function identities
is the valence formula \cite[Theorem 4.1.4, p.98]{Ra}\omycite{Ra}.
If $f\ne0$ is a modular function on $\Gamma$ and $\mathscr{F}$ is any
fundamental set for $\Gamma$ then
\beq
\sum_{z\in\mathscr{F}} \ORD(f,z,\Gamma) = 0.
\omylabel{eq:valform}
\eeq

\subsection{Eta-product identities}
\omylabel{subsec:etaprods}
We will consider eta-products of the form
\begin{equation}
f(\tau) = \prod_{d\mid N} \eta(d\tau)^{m_{d}},
\omylabel{eq:etapdef}
\end{equation}
where $N$ is a positive integer,  each $d>0$ and $m_{d}\in\Z$.

\subsubsection*{Modularity}
Newman \cite{Ne59}\omycite{Ne59} has found necessary and sufficient conditions
under which an eta-product is a modular function on $\Gamma_0(N)$.
\begin{theorem}[{\cite[Theorem 4.7]{Ne59}\omycite{Ne59}}]
\omylabel{thm:etamodthm}
The function $f(\tau)$ (given in \omyeqn{etapdef}) is a modular function
on $\Gamma_0(N)$ if and only if
\begin{enumerate}
\item
$\displaystyle\sum_{d\mid N} m_d = 0$,
\item
$\displaystyle\sum_{d\mid N} d m_d \equiv0\pmod{24}$,
\item
$\displaystyle\sum_{d\mid N} \frac{N m_d}{d} \equiv0\pmod{24}$, and
\item
$\displaystyle\prod_{d\mid N} d^{|m_d|}$ is a square.
\end{enumerate}
\end{theorem}

\subsubsection*{Orders at cusps}
Ligozat \cite{Li75}\omycite{Li75} has computed the invariant order of an eta-product
 at the cusps of $\Gamma_0(N)$.
\begin{theorem}[{\cite[Theorem 4.8]{Li75}\omycite{Li75}}]
\omylabel{thm:ordthm}
If the eta-product $f(\tau)$ (given in \omyeqn{etapdef})   is a modular function
on $\Gamma_0(N)$, then its order at the cusp $\zeta=\frac{b}{c}$
(assuming $(b,c)=1$) is
\begin{equation}
\ord(f(\tau);\zeta)=\sum_{d\mid N} \frac{(d,c)^2 m_d}{24d}.
\label{eq:ecord}
\end{equation}
\end{theorem}

Chua and Lang \cite{Ch-La04}\omycite{Ch-La04}
have found a set of inequivalent
cusps for $\Gamma_0(N)$.
\begin{theorem}[{\cite[p.354]{Ch-La04}\omycite{Ch-La04}}]
\omylabel{thm:chualang}
Let N be a positive integer and for each positive divisor $d$ of $N$ let
$e_d = (d,N/d)$. Then the set
\beqs
\Delta = {\underset{d\mid N}{\cup}} \, S_d
\eeqs
is a complete set of inequivalent cusps of $\Gamma_0(N)$ where
$$
S_d = \{ x_i/d\,:\,(x_i,d)=1,\quad 0\le x_i\le d-1,\quad x_i\not\equiv
x_j \pmod{e_d}\}.
$$
\end{theorem}
Biagioli \cite{Bi89}\omycite{Bi89} has found the fan width of the cusps of
$\Gamma_0(N)$.
\begin{lemma}[{\cite[Lemma 4.2]{Bi89}\omycite{Bi89}}]
\omylabel{lem:fanw}
If $(r,s)=1$, then the fan width of $\Gamma_0(N)$ at $\frac{r}{s}$
is
$$
\kappa\left(\Gamma_0(N); \frac{r}{s}\right) = \frac{N}{(N,s^2)}.
$$
\end{lemma}

\subsubsection*{An application of the valence formula}
Since eta-products have no zeros or poles in $\uhp$ the following result
follows easily from the valence formula \omyeqn{valform}.
\begin{theorem}
\omylabel{thm:valcor}
Let $f_1(\tau)$, $f_2(\tau)$, \dots, $f_n(\tau)$ be eta-products that
are modular functions on $\Gamma_0(N)$. Let $\mathcal{S}_N$ be a set of inequivalent
cusps for $\Gamma_0(N)$. Define the constant
\beq
B = \sum_{\substack{\zeta\in\mathcal{S}_N\\ \zeta\ne \infty}}
        \mbox{min}
        (\left\{\ORD(f_j,\zeta,\Gamma_0(N))\,:\, 1 \le j \le n\right\}),
\omylabel{eq:Bdef}
\eeq
and consider
\beq
g(\tau) := \alpha_1 f_1(\tau) + \alpha_2 f_2(\tau) + \cdots + \alpha_n f_n(\tau),
\omylabel{eq:gdef}
\eeq
where each $\alpha_j\in\mathbb{C}$. Then
$$
g(\tau) \equiv 0
$$
if and only if
\beq
\ORD(g(\tau), \infty, \Gamma_0(N)) > -B.
\omylabel{eq:ORDBineq}
\eeq
\end{theorem}

\noindent
\textit{An algorithm for proving eta-product identities.} 

        \vskip 10pt\noindent
{\it\footnotesize STEP 0}. \quad  Write the identity in the following
form: 
\begin{equation}
    \alpha_1 f_1(\tau) + \alpha_2 f_2(\tau) + \cdots + \alpha_n f_n(\tau)  = 0,
\omylabel{eq:fid}
\end{equation}
where each $\alpha_i\in\C$ and each $f_i(\tau)$ is an eta-product of
level $N$.

        \vskip 10pt\noindent
{\it\footnotesize STEP 1}. \quad  Use Theorem \omythm{etamodthm} to check that
$f_j(\tau)$ is a modular function on $\Gamma_0(N)$ for each
$1 \le j \le n$.

        \vskip 10pt\noindent
{\it\footnotesize STEP 2}. \quad  Use Theorem \omythm{chualang} to
find a set $\mathcal{S}_N$ of inequivalent cusps for $\Gamma_0(N)$ and the
fan width of each cusp.

        \vskip 10pt\noindent
{\it\footnotesize STEP 3}. \quad  Use Theorem \omythm{ordthm} to
calculate the order of each eta-product
$f_j(\tau)$ at each cusp of $\Gamma_0(N)$.

        \vskip 10pt\noindent
{\it\footnotesize STEP 4}. \quad  Calculate
        $$
        B =
        \sum_{\substack{\zeta\in\mathcal{S}_N\\ \zeta\ne \infty}}
        \mbox{min}
        (\left\{\ORD(f_j,\zeta,\Gamma_0(N))\,:\, 1 \le j \le n\right\} ).
        $$

        \vskip 10pt\noindent
{\it\footnotesize STEP 5}. \quad  Show that
        $$
        \ORD(g(\tau),\infty,\Gamma_0(N)) > -B
        $$
        where
        $$
        g(\tau) = \alpha_1 f_1(\tau) + \alpha_2 f_2(\tau) +
        \cdots + \alpha_n f_n(\tau).
        $$
        Theorem \omythm{valcor} then implies that $g(\tau)\equiv0$ and
        hence the eta-product identity  \omyeqn{fid}.

The third author has written a \textsc{MAPLE} package
called \texttt{ETA} which implements this algorithm. See
\begin{center}
\url{http://qseries.org/fgarvan/qmaple/ETA/}
\end{center}

\subsubsection*{A modular equation}
Define
\begin{align}
t&:=t(\tau):=\frac{\eta(\tau)^2\eta(10\tau)^4}{\eta(2\tau)^4\eta(5\tau)^2}
\omylabel{eq:deft}\\
&=q-2\,{q}^{2}+3\,{q}^{3}-6\,{q}^{4}+11\,{q}^{5}-16\,{q}^{6}+24\,{q}^{7
}-38\,{q}^{8}+57\,{q}^{9}-82\,{q}^{10}+117\,{q}^{11}+\cdots.
\nonumber
\end{align}
We note that $t(\tau)$ is a Hauptmodul for $\Gamma_0(10)$ 
\cite{Ma09}\omycite{Ma09}.
As an application of our algorithm we prove the following theorem which will be needed
later.
\begin{theorem}
\omylabel{thm:modeq}
Let 
\begin{align}
\sigma_0(\tau)&=-t,\omylabel{eq:sig0}\\
\sigma_1(\tau)&=-5t^2+2\cdot5t,\omylabel{eq:sig1}\\
\sigma_2(\tau)&=-5^2t^3+2\cdot5^2t^2-7\cdot5t,\omylabel{eq:sig2}\\
\sigma_3(\tau)&=-5^3t^4+2\cdot5^3t^3-7\cdot5^2t^2+12\cdot5t,\omylabel{eq:sig3}\\
\sigma_4(\tau)&=-5^4t^5+2\cdot5^4t^4-7\cdot5^3t^3+12\cdot5^2t^2-11\cdot5t,
\omylabel{eq:sig4}
\end{align}
where $t=t(\tau)$ is defined in \eqn{deft}.
Then
\beq
t(\tau)^5 + \sum_{j=0}^4 \sigma_j(5\tau) \, t(\tau)^j = 0.
\omylabel{eq:modeq}
\eeq
\end{theorem}
\begin{proof}
From Theorem \omythm{etamodthm} we find that $t(\tau)$ is a modular
function on $\Gamma_0(10)$ and $t(5\tau)$ is a modular function
on $\Gamma_0(50)$. Hence each term on the left side of \omyeqn{modeq}
is a modular function on $\Gamma_0(50)$. For convenience
we divide by $t(\tau)^5$ and let
\beq
g(\tau) = 1 + \sum_{j=0}^4 \sigma_j(5\tau) \, t(\tau)^{j-5}.
\omylabel{eq:gsum}
\eeq
From Theorem \omythm{chualang}, Lemma \omylem{fanw} and Theorem \omythm{ordthm}
we have the
following table of fan widths for the cusps of $\Gamma_0(50)$, with
the orders and invariant orders of both $t(\tau)$ and $t(5\tau)$.

$$
\begin{array}{|c|c|c|c|c|c|c|c|c|c|c|c|c|}
\noalign{\hrule}
\zeta &0&1/2&1/5&2/5&3/5&4/5&1/10&3/10&7/10&9/10&1/25&1/50 \\
\noalign{\hrule}
\kappa(\Gamma_0(50),\zeta) &50& 25& 2& 2& 2& 2& 1& 1& 1& 1& 2& 1\\
\noalign{\hrule}
\ord(t(\tau),\zeta)&0& -1/5& 0& 0& 0& 0& 1& 1& 1& 1& 0& 1\\
\noalign{\hrule}
\ORD(t(\tau),\zeta,\Gamma_0(50)&0& -5& 0& 0& 0& 0& 1& 1& 1& 1& 0& 1\\
\noalign{\hrule}
\ord(t(5\tau)&0& -1/25& 0& 0& 0& 0& -1& -1& -1& -1& 0& 5\\
\noalign{\hrule}
\ORD(t(5\tau),\zeta,\Gamma_0(50)&0& -1& 0& 0& 0& 0& -1& -1& -1& -1& 0& 5\\
\noalign{\hrule}
\end{array}
$$
Expanding the right side of \omyeqn{gsum} gives $16$ terms of the
form $t(5\tau)^k t(\tau)^{j-5}$ with $1\le k
\le j+1$ where $0 \le j \le 4$, together with $(k,j)=(0,5)$.
We calculate the order of each term at each cusp $\zeta$ of $\Gamma_0(50)$,
and thus giving  lower bounds for $\ORD(g(\tau),\zeta,\Gamma_(50)$ at each cusp
in the following.
$$
\begin{array}{|c|c|c|c|c|c|c|c|c|c|c|c|c|}
\noalign{\hrule}
\zeta &0&1/2&1/5&2/5&3/5&4/5&1/10&3/10&7/10&9/10&1/25&1/50 \\
\noalign{\hrule}
\ORD(g(\tau),\zeta,\Gamma_0(50))\ge &0& 0& 0& 0& 0& 0& -6& -6& -6& -6& 0& 0 \\
\noalign{\hrule}
\end{array}
$$
Thus the constant $B$ in Theorem \omythm{valcor} is $B=-24$. It suffices to show
that
$$
\ORD(g(\tau),\infty,\Gamma_0(50))> 24.
$$
This is easily verified. Thus by Theorem \omythm{valcor} we have $g(\tau) \equiv 0$ and
the result follows.
\end{proof}

\subsection{The $U_p$ operator}
\omylabel{subsec:Upop}

Let $p$ prime and
$$
f = \sum_{m=m_0}^\infty a(m) q^m
$$
be a formal Laurent series. We
define $U_p$ by
\beq
U_p(f) := \sum_{p m \ge m_0}  a(p m) q^m. 
\omylabel{eq:Updeffls}
\eeq
If $f$ is a modular function (with $q=\exp(2\pi i\tau)$),
\beq
U_p(f) = \frac{1}{p} \sum_{j=0}^p \stroke{f}{\MAT{1/p}{j/p}{0}{1}}
 = \frac{1}{p} \sum_{j=0}^p f\left(\frac{\tau +j}{p}\right).
\omylabel{eq:Updef}
\eeq
By \cite[Lemma 7, p.138]{At-Le70}\omycite{At-Le70} we have
\begin{theorem}
\omylabel{thm:ALUpthm}
Let $p$ be prime. If $f$ is a modular function on $\Gamma_0(pN)$
and $p\mid N$, then $U_p(f)$ is a modular function
on $\Gamma_0(N)$.
\end{theorem}
Gordon and Hughes \cite[Theorem 4, p.336]{GoHu81}\omycite{GoHu81} have found
lower bounds for the invariant orders of $U_p(f)$ at cusps.
Let $\nu_p(n)$ denote the $p$-adic order of an integer $n$; i.e.\ the
highest power of $p$ that divides $n$.

\begin{theorem}[{\cite[Theorem 4]{GoHu81}\omycite{GoHu81}}]
\omylabel{thm:UpLB}
Suppose $f(\tau)$ is a modular function on $\Gamma_0(pN)$, where $p$ is
prime and $p\mid N$. Let $r=\frac{\beta}{\delta}$ be a cusp of
$\Gamma_0(N)$, where $\delta\mid N$ and $(\beta,\delta)=1$. Then
$$
\ORD(U_p(f), r, \Gamma_0(N)) \ge
\begin{cases}
\frac{1}{p} \ORD(f,r/p,\Gamma_0(pN)) & \mbox{if $\nu_p(\delta)\ge \frac{1}{2} \nu_p(N)$}\\
\ORD(f,r/p,\Gamma_0(pN)) & \mbox{if $0 < \nu_p(\delta)< \frac{1}{2} \nu_p(N)$}\\
\umin{0 \le k \le p-1} \ORD(f,(r+k)/p,\Gamma_0(pN))
&\mbox{if $\nu_p(\delta)=0$}.
\end{cases}
$$
\end{theorem}

Theorems \omythm{valcor}, \omythm{ALUpthm} and \omythm{UpLB} give
the following algorithm.

\subsubsection*{An algorithm for proving $U_p$ eta-product identities} \hphantom{X}

        \vskip 10pt\noindent
{\it\footnotesize STEP 0}. \quad  Write the identity in the form
\begin{equation}
{U_p}\left(\alpha_1 g_1(\tau) + \alpha_2 g_2(\tau) + \cdots + \alpha_n g_k(\tau)
\right)=
\alpha_1 f_1(\tau) + \alpha_2 f_2(\tau) + \cdots + \alpha_n f_n(\tau),
\omylabel{eq:Upfid}
\end{equation}
where $p$ is prime, $p\mid N$, each $g_j(\tau)$ is an eta-product and
a modular function on $\Gamma_0(pN)$, and each $f_j(\tau)$
is an eta-product and modular function on $\Gamma_0(N)$.

        \vskip 10pt\noindent
{\it\footnotesize STEP 1}. \quad  Use Theorem \omythm{etamodthm} to check that
$f_j(\tau)$ is a modular function on $\Gamma_0(N)$ for each
$1 \le j \le n$, and $g_j(\tau)$ is a modular function on $\Gamma_0(pN)$
for each $1 \le j \le k$.

        \vskip 10pt\noindent
{\it\footnotesize STEP 2}. \quad  Use Theorem \omythm{chualang} to
find a set $\mathcal{S}_N$ of inequivalent cusps for $\Gamma_0(N)$ and the
fan width of each cusp.

        \vskip 10pt\noindent
{\it\footnotesize STEP 3a}. \quad  Compute $\ORD(f_j,\zeta,\Gamma_0(N))$
for each $j$ at each cusp $\zeta$ of $\Gamma_0(N)$ apart from $\infty$.

        \vskip 10pt\noindent
{\it\footnotesize STEP 3b}. \quad  Use Theorem \omythm{UpLB}
to find lower bounds $L(g_j,\zeta,N)$ for
$$
\ORD({U_p}(g_j), \zeta, \Gamma_0(N))
$$
for each cusp $\zeta$ of $\Gamma_0(N)$, and each $1 \le j \le k$.

        \vskip 10pt\noindent
{\it\footnotesize STEP 4}. \quad  Calculate
\beq
        B =
        \sum_{\substack{\zeta\in\mathcal{S}_N\\ \zeta\ne \infty}}
        \mbox{min}
(\left\{\ORD(f_j,\zeta,\Gamma_0(N))\,:\, 1 \le j \le n\right\} \cup 
\{L(g_j,\zeta,N)\,:\, 1 \le j \le k \}).
\omylabel{eq:UpB}
\eeq

        \vskip 10pt\noindent
{\it\footnotesize STEP 5}. \quad  Show that
        $$
        \ORD(h(\tau),\infty,\Gamma_0(N)) > -B
        $$
        where
        $$
        h(\tau) = 
{U_p}\left(\alpha_1 g_1(\tau) + \alpha_2 g_2(\tau) + \cdots + \alpha_n g_k(\tau)
        \right)
        - (\alpha_1 f_1(\tau) + \alpha_2 f_2(\tau) +
        \cdots + \alpha_n f_n(\tau)).
        $$
        Theorem \omythm{valcor} then implies that $h(\tau)\equiv0$ and
        hence the $U_p$ eta-product identity  \omyeqn{Upfid}.

The third author has included an implementation of this
algorithm in his \texttt{ETA} \textsc{MAPLE} package.

As an application of our algorithm we sketch the proof of
\beq
{U_5}(g) = 5\,f_1(\tau) + 2\,f_2(\tau),
\omylabel{eq:Upalgeg}
\eeq
where
$$
g(\tau)={\frac {  \eta \left( 50\,\tau \right)   ^{5}  \eta  \left( 5\,\tau\right)    ^{4}  \eta \left( 4\,\tau  \right)   ^{3}  \eta \left( 2\,\tau \right)   ^{3} }{  \eta \left( 100\,\tau \right)   ^{3}  \eta  \left( 25\,\tau \right)   ^{2}  \eta \left( 10\,\tau  \right)   ^{8}  \eta \left( \tau \right)   ^{2}}},
$$
$$
f_1(\tau) = {\frac {  \eta \left( 10\,\tau \right)  ^{8}  \eta  \left( \tau \right)  ^{4}}{  \eta \left( 5\,\tau  \right)  ^{4}  \eta \left( 2\,\tau \right)  ^{8} }},\quad
f_2(\tau)= {\frac {  \eta \left( 10\,\tau \right)  ^{5}  \eta  \left( \tau \right)  ^{2}}{  \eta \left( 20\,\tau  \right)  ^{3}  \eta \left( 5\,\tau \right)  ^{2} \eta \left( 4\,\tau \right) \eta \left( 2\,\tau \right) }}.
$$
We use Theorem \omythm{etamodthm} to check that
$f_j(\tau)$ is a modular function on $\Gamma_0(20)$ for each
$1 \le j \le 2$, and $g(\tau)$ is a modular function on $\Gamma_0(100)$.
We use Theorem \omythm{chualang} to
find a set $\mathcal{S}_{20}$ of inequivalent cusps for $\Gamma_0(20)$ and the
fan width of each cusp. By Theorems \omythm{chualang}, \omythm{ordthm}
and Lemma \omylem{fanw} we have the following table of orders.
$$
\begin{array}{|c|c|c|c|c|c|c|}
\noalign{\hrule}
\zeta &0&1/2&1/4&1/5&1/10&1/20\\
\noalign{\hrule}
\ORD(f_1(\tau),\zeta,\Gamma_0(20))& 0& -2& -2& 0& 2& 2\\
\noalign{\hrule}
\ORD(f_2(\tau),\zeta,\Gamma_0(20))& 1& 0& -1& 0& 1& -1\\
\noalign{\hrule}
\end{array}
$$
Using
Theorems \omythm{chualang}, \omythm{ordthm}, \thm{UpLB} and some calculation
we have the following table of lower bounds $L(g,\zeta,20)$.
$$
\begin{array}{|c|c|c|c|c|c|c|}
\noalign{\hrule}
\zeta &0&1/2&1/4&1/5&1/10&1/20\\
\noalign{\hrule}
\ORD({U_5}(g), \zeta, \Gamma_0(20)) \ge  & 0& -2& -2& -1/5& 3/5& -6/5\\
\noalign{\hrule}
\end{array}
$$
Thus the constant $B$ in \omyeqn{UpB} is $B=-18/5$.
It suffices to show
that
$$
\ORD(h(\tau),\infty,\Gamma_0(20))\ge 4,
$$
where
        $$
        h(\tau) = {U_5}(g)
        - (5 f_1(\tau) + 2 f_2(\tau)).
        $$
This is easily verified. Thus by Theorem \omythm{valcor} we have $h(\tau) \equiv 0$ and
the result \omyeqn{Upalgeg} follows.

\subsection{Generalized eta-functions}
\omylabel{subsec:geneta}
The generalized Dedekind eta function is defined to be
\begin{equation}
\eta_{\delta,g}(\tau) =q^{\frac{\delta}{2} P_2(g/\delta) }
\prod_{m\equiv \pm g\pmod{\delta}} (1 - q^m),
\omylabel{eq:Geta}
\end{equation}
where $P_2(t) = \{t\}^2 - \{t\} + \tfrac16$ is the second periodic Bernoulli polynomial,
$\{t\}=t-[t]$ is the fractional part of $t$, $g,\delta,m\in \mathbb{Z}^{+}$
and $0 < g < \delta$. The function $\eta_{\delta,g}(\tau)$ is a modular function
on $\SLZ$ with a multiplier system.
Let $N$ be a fixed positive integer.
A generalized Dedekind eta-product of level $N$ has the form
\begin{equation} \omylabel{eq:fdef}
f(\tau) = \prod_{\substack{\delta\mid N\\ 0 < g < \delta}}
                 \eta_{\delta,g}^{r_{\delta,g}}(\tau),
\end{equation}
where
\begin{equation} \omylabel{r-dg}
     r_{\delta,g} \in
     \begin{cases}
      \frac{1}{2}\Z & \mbox{if $g=\delta/2$},\\
      \Z & \mbox{otherwise}.
      \end{cases}
\end{equation}
Robins \cite{Ro94}\omycite{Ro94} has found sufficient conditions under which a
generalized eta-product is a modular function on $\Gamma_1(N)$.
\begin{theorem}[{\cite[Theorem 3]{Ro94}\omycite{Ro94}}]
\omylabel{thm:gepmodfunc}
The function $f(\tau)$, defined in \omyeqn{fdef}, is a
modular function on $\Gamma_1(N)$ if
\begin{enumerate}
\item[(i)]
$\displaystyle\sum_{\substack{\delta \mid N \\ g}}
\delta P_2(\textstyle{\frac{g}{\delta}}) r_{\delta,g} \equiv 0 \pmod{2}$,
and
\item[(ii)]
$\displaystyle\sum_{\substack{\delta \mid N \\ g} }
\frac{N}{\delta} P_2(0) r_{\delta,g} \equiv 0 \pmod{2}$.
\end{enumerate}
\end{theorem}

Cho, Koo and Park \cite{Ch-Ko-Pa}\omycite{Ch-Ko-Pa} have found a set of
inequivalent
cusps for $\Gamma_1(N) \cap \Gamma_0(mN)$.
The group $\Gamma_1(N)$ corresponds to the case $m=1$.
\begin{theorem}[{\cite[Corollary 4, p.930]{Ch-Ko-Pa}\omycite{Ch-Ko-Pa}}]
\omylabel{thm:chokoopark}
Let $a$, $c$, $a'$, $c\in\Z$ with $(a,c)=(a',c')=1$.
\begin{enumerate}
\item[(i)]
The cusps $\frac{a}{c}$ and $\frac{a'}{c'}$ are equivalent mod $\Gamma_1(N)$
if and only if
$$
\begin{pmatrix}
a' \\ c'
\end{pmatrix}
\equiv \pm
\begin{pmatrix}
 a + nc\\
 c
\end{pmatrix}
\pmod{N}
$$
for some integer $n$.
\item[(ii)]
The following is a complete set of inequivalent cusps mod $\Gamma_1(N)$.
\begin{align*}
\mathcal{S} &= \left\{ \frac{y_{c,j}}{x_{c,i}} \,:\, 0 < c \mid N,\,
  0 < s_{c,i},\, a_{c,j} \le N,\, (s_{c,i},N)=(a_{c,j},N)=1,\right.\\
     &\qquad s_{c,i}=s_{c,i'} \iff s_{c,1}\equiv\pm s_{c',i'}
                  \pmod{{\textstyle \frac{N}{c}}},\\
    &\qquad a_{c,j}=a_{c,j'} \iff
\begin{cases}
a_{c,j}\equiv\pm a_{c,j'} \pmod{c}, &\mbox{if $c =\frac{N}{2}$ or $N$},\\
a_{c,j}\equiv a_{c,j'} \pmod{c}, &\mbox{otherwise},
\end{cases} \\
&\left. \vphantom{ \frac{y_{c,j}}{x_{c,i}} } x_{c,i}, y_{c,j}\in\Z\,
\mbox{chosen s.th.}\,
x_{c,i}\equiv c s_{c,i},\, y_{c,j}\equiv a_{c,j}\pmod{N},
\, (x_{c,i},y_{c,j})=1\right\},
\end{align*}
\item[(iii)]
and the fan width of the cusp $\frac{a}{c}$ is given by
$$
\kappa({\textstyle \frac{a}{c}},\Gamma_1(N)) =
\begin{cases}
1, & \mbox{if $N=4$ and $(c,4)=2$},\\
\frac{N}{(c,N)}, & \mbox{otherwise}.
\end{cases}
$$
\end{enumerate}
\end{theorem}
In this theorem, it is understood as usual that the fraction $\frac{\pm1}{0}$
corresponds to $\infty$.

Robins \cite{Ro94}\omycite{Ro94} has calculated the invariant order of
$\eta_{\delta,g}(\tau)$ at any cusp.
This gives a method for calculating the invariant
order at any cusp of a generalized eta-product.
\begin{theorem}[\cite{Ro94}\omycite{Ro94}]
\omylabel{thm:cuspord1}
The order at the cusp $\zeta=\frac{a}{c}$ (assuming $(a,c)=1$) of the
generalized eta-function $\eta_{\delta,g}(\tau)$ (defined in \omyeqn{Geta}
and assuming $0 < g < \delta$)
is
\beq
\ord(\eta_{\delta,g}(\tau);\zeta) = \frac{\varepsilon^2}{2\delta}\,
P_2\left(\frac{ag}{\varepsilon}\right),
\omylabel{eq:ordetadg}
\eeq
where $\varepsilon=(\delta,c)$.
\end{theorem}

\subsubsection*{An algorithm for proving generalized eta-product identities} \hphantom{X}

        \vskip 10pt\noindent
We note that the analog of Theorem \thm{valcor} holds for generalized
eta-products which are modular functions on $\Gamma_1(N)$, and follows easily
from the valence formula \omyeqn{valform}.
\begin{theorem}[{\cite[Cor.2.5]{Fr-Ga19}\omycite{Fr-Ga19}}]
\omylabel{thm:valcor1}
Let $f_1(\tau)$, $f_2(\tau)$, \dots, $f_n(\tau)$ be generalized eta-products that
are modular functions on $\Gamma_1(N)$. Let $\mathcal{S}_N$ be a set of inequivalent
cusps for $\Gamma_1(N)$. Define the constant
\beq
B = \sum_{\substack{s\in\mathcal{S}_N\\s\ne \infty}}
        \mbox{min}
        (\left\{\ORD(f_j,s,\Gamma_1(N))\,:\, 1 \le j \le n\right\} \cup \{0\}),
\omylabel{eq:Bdef1}
\eeq
and consider
\beq
g(\tau) := \alpha_1 f_1(\tau) + \alpha_2 f_2(\tau) + \cdots + \alpha_n f_n(\tau) + 1,
\omylabel{eq:gdef1}
\eeq
where each $\alpha_j\in\mathbb{C}$. Then
$$
g(\tau) \equiv 0
$$
if and only if
\beq
\ORD(g(\tau), \infty, \Gamma_1(N)) > -B.
\omylabel{eq:ORDBineq1}
\eeq
\end{theorem}

The algorithm for proving
generalized eta-product identities is completely analogous to the method
for proving eta-product identities described in Section \subsect{etaprods}.
To prove an identity in the form
$$
\alpha_1 f_1(\tau) + \alpha_2 f_2(\tau) + \cdots + \alpha_n f_n(\tau) + 1
=0,
$$
the algorithm simply involves calculating the constant $B$ in \omyeqn{Bdef1}
and then calculating enough coefficients to show that the inequality
\omyeqn{ORDBineq1} holds.
A more complete description is given in \cite{Fr-Ga19}\omycite{Fr-Ga19}.

The third author has written a \textsc{MAPLE} package
called \texttt{thetaids} which implements this algorithm. See
\begin{center}
\url{http://qseries.org/fgarvan/qmaple/thetaids/}
\end{center}


\section{The rank parity function modulo powers of $5$}
\omylabel{sec:rankparity5}
\subsection{A Generating Function}
\omylabel{subsec:genfunc}

In this section we prove an identity for the generating function of
$$
a_f(5n-1) + a_f(n/5),
$$
where it is understood that $a_f(n)=0$ if $n$ is not a non-negative integer.
Our proof depends on some results of Mao \cite{Mao13}\omycite{Mao13} who found $5$-dissection
results for the rank modulo $10$.
\begin{theorem}
\omylabel{thm:af5thm}
\beq
\sum_{n=0}^\infty (a_f(5n-1) + a_f(n/5)) q^n =
\frac{J_2^4J_{10}^2}
     {J_1J_4^3J_{20}}
    -4q\frac{J_1^2J_4^3J_5J_{20}}
            {J_2^5J_{10}}.
\omylabel{eq:af5id}
\eeq
\end{theorem}
\begin{proof}
From Watson \cite[p.64]{Wa36a}\omycite{Wa36a} we have           
\begin{align}
f(q)&=\frac{2}{(q;q)_\infty}\sum_{n=-\infty}^{\infty}\frac{(-1)^nq^{n(3n+1)/2}}{1+q^n}.
\omylabel{eq:fqid}
\end{align}

We find that
\begin{align}   
\sum_{n=-\infty}^{\infty}\frac{(-1)^nq^{n(3n+1)/2+4n}}{1+q^{5n}}&
=\sum_{n=-\infty}^{\infty}\frac{(-1)^nq^{n(3n+1)/2}}{1+q^{5n}},
\omylabel{eq:lam5ids}\\
\sum_{n=-\infty}^{\infty}\frac{(-1)^nq^{n(3n+1)/2+3n}}{1+q^{5n}}&
=\sum_{n=-\infty}^{\infty}\frac{(-1)^nq^{n(3n+1)/2+n}}{1+q^{5n}}.
\nonumber
\end{align}
By \cite[Lemma 3.1]{Mao13}\omycite{Mao13} we have
\begin{align}    
\sum_{n=-\infty}^{\infty}\frac{(-1)^nq^{n(3n+1)/2}}{1+q^{5n}}
&=P(q^5,-q^5;q^{25})-\frac{P(q^{10},-q^5;q^{25})}{q^3}
+\frac{(q;q)_\infty}{J_{25}}\sum_{n=-\infty}^{\infty}\frac{(-1)^nq^{75n(n+1)/2+5}}{1+q^{25n+5}},
\omylabel{eq:mao1}\\
\sum_{n=-\infty}^{\infty}\frac{(-1)^nq^{n(3n+1)/2+n}}{1+q^{5n}}
&=P(q^{10},-q^{10};q^{25})-q^3P(q^5,-q^{10};q^{25})
-\frac{(q;q)_\infty}{J_{25}}\sum_{n=-\infty}^{\infty}\frac{(-1)^nq^{75n(n+1)/2+8}}{1+q^{25n+10}},
\omylabel{eq:mao2}\\
\sum_{n=-\infty}^{\infty}\frac{(-1)^nq^{n(3n+1)/2+2n}}{1+q^{5n}}
&=\frac{P(q^5,-1;q^{25})}{q^6}-\frac{P(q^{10},-1;q^{25})}{q^9}
-\frac{(q;q)_\infty}{J_{25}}
\sum_{n=-\infty}^{\infty}\frac{(-1)^nq^{25n(3n+1)/2-1}}{1+q^{25n}},
\omylabel{eq:mao3}
\end{align}
where
\beq
P(a,b;q)=\frac{[a,a^2;q]_\infty (q;q)_\infty^2}{[b/a,ab,b;q]_\infty}.
\omylabel{eq:Pabqdef}
\eeq

From \omyeqn{fqid}-\omyeqn{mao3}, and noting that $P(q^5,-q^5;q^{25})=P(q^{10},-q^{10};q^{25})$ 
we have
\begin{align}
f(q)=&\frac{2}{(q;q)_\infty}\sum_{n=-\infty}^{\infty}\frac{(-1)^nq^{n(3n+1)/2}}{1+q^n}
\omylabel{eq:flamid}\\
=&\frac{2}{(q;q)_\infty}\sum_{n=-\infty}^{\infty}\frac{(-1)^nq^{n(3n+1)/2}(1-q^n+q^{2n}-q^{3n}+q^{4n})}{1+q^{5n}}
\nonumber\\
=&\frac{2}{(q;q)_\infty}\sum_{n=-\infty}^{\infty}\frac{(-1)^nq^{n(3n+1)/2}(2-2q^n+q^{2n})}{1+q^{5n}}
\nonumber\\
=&\frac{2}{J_1}\bigg\{2q^3P(q^5,-q^{10};q^{25})-\frac{2P(q^{10},-q^5;q^{25})}{q^3}+\frac{P(q^5,-1;q^{25})}{q^6}-\frac{P(q^{10},-1;q^{25})}{q^9}\bigg\}
\nonumber\\
&+\frac{4}{J_{25}}\sum_{n=-\infty}^{\infty}\frac{(-1)^nq^{75n(n+1)/2+5}}{1+q^{25n+5}}
+\frac{4}{J_{25}}\sum_{n=-\infty}^{\infty}\frac{(-1)^nq^{75n(n+1)/2+8}}{1+q^{25n+10}}
-\frac{1}{q}f(q^{25}).
\nonumber
\end{align}

We let 
\beq            
g(q)=\frac{2}{J_1}\bigg\{2q^3P(q^5,-q^{10};q^{25})-\frac{2P(q^{10},-q^5;q^{25})}{q^3}
+\frac{P(q^5,-1;q^{25})}{q^6}-\frac{P(q^{10},-1;q^{25})}{q^9}\bigg\},
\omylabel{eq:gq1}
\eeq            
and
write the $5$-dissection of $g(q)$ as
\begin{align}
g(q)=g_0(q^5) +q\,g_1(q^5)+q^2\,g_2(q^5)+q^3\,g_3(q^5)+q^4\,g_4(q^5).
\omylabel{eq:gq2}
\end{align}

From \omyeqn{fqid}, \omyeqn{gq1} and \omyeqn{gq2}, replacing $q^5$ by $q$, we have
\beq                  
\sum_{n=0}^\infty a_f(5n+4) q^n =-\frac{1}{q}f(q^5)+g_4(q),
\omylabel{eq:f4g}
\eeq               
after dividing both sides by $q^4$ and replacing $q^5$ by $q$.

The 5-dissection of $J_1$ is well-known
\beq               
J_1=J_{25}\bigg(B(q^5)-q-q^2\frac{1}{B(q^5)}\bigg),
\omylabel{eq:gq3}
\eeq               
where
\beqs 
B(q)=\frac{J_{2,5}}{J_{1,5}}.
\eeqs
See for example \cite[Lemma (3.18)]{Ga88b}\omycite{Ga88b}.

From \omyeqn{gq1}, \omyeqn{gq2} and \omyeqn{gq3}
\begin{align}
&J_{25}\,(g_0(q^5) +q\,g_1(q^5)+q^2\,g_2(q^5)+q^3\,g_3(q^5)
+q^4\,g_4(q^5)) \bigg(B(q^5)-q-q^2\frac{1}{B(q^5)}\bigg)
\omylabel{eq:gq4}\\
&=4q^3P(q^5,-q^{10};q^{25})-\frac{4P(q^{10},-q^5;q^{25})}{q^3}+\frac{2P(q^5,-1;q^{25})}{q^6}
-\frac{2P(q^{10},-1;q^{25})}{q^9}.
\nonumber            
\end{align}

By expanding the left side of \omyeqn{gq4} and comparing both sides 
according to the residue of the exponent of $q$ modulo 5, 
we obtain 5 equations:
\begin{align}
&B(q)g_0-q^5g_4-\frac{q^5}{B(q)}g_3=0,
\omylabel{eq:eq1}\\
&B(q)g_1-g_0-\frac{q}{B(q)}g_4=-\frac{2P(q^{2},-1;q^5)}{q^{2}J_{5}},
\omylabel{eq:eq2}\\
&B(q)g_2-g_1-\frac{1}{B(q)}g_0=-\frac{4P(q^{2},-q;q^5)}{qJ_{5}},
\omylabel{eq:eq3}\\
&B(q)g_3-g_2-\frac{1}{B(q)}g_1=\frac{4P(q,-q^{2};q^5)}{J_{5}},
\omylabel{eq:eq4}\\
&B(q)g_4-g_3-\frac{1}{B(q)}g_2=\frac{2P(q,-1;q^5)}{q^{2}J_{5}},
\omylabel{eq:eq5}
\end{align}
where $g_j=g_j(q)$ for $0\le j \le4$.

Solving these equations we find that
\begin{align}
g_4(q)
&=\frac{1}{J_5(B^5-11q-q^2/B^5)}\bigg(\frac{2}{q^2}X_2B^4
                            -\frac{2}{q}X_1B^{-4}+4X_4B^3+4X_3B^{-3}
\omylabel{eq:g4q}\\
& -\frac{8}{q}X_3B^2+8qX_4B^{-2}-\frac{6}{q^2}X_1B-\frac{6}{q}X_2B^{-1}\bigg),
\nonumber             
\end{align}
where $B:=B(q)$ and
\begin{align*}
X_1=&P(q^2,-1;q^5)=\frac{q^2J_{1,10}J_{2,10}^3J_{3,10}^3J_{5,10}^2}{2J_{10}^6J_{4,10}},\quad
X_2=P(q,-1;q^5)=\frac{qJ_{1,10}^3J_{3,10}J_{4,10}^3J_{5,10}^2}{2J_{10}^6J_{2,10}},\\
X_3=&P(q^2,-q;q^5)=\frac{qJ_{1,10}^3J_{3,10}^2J_{4,10}^2J_{5,10}}{J_{10}^6},\quad
X_4=P(q,-q^2;q^5)=\frac{J_{1,10}^2J_{2,10}^2J_{3,10}^3J_{5,10}}{J_{10}^6}.
\end{align*}

The following identity is also well-known
\beq                  
\frac{J_1^6}{J_5^6}=B^5-11q-q^2\frac{1}{B^5}.
\omylabel{eq:idb}
\eeq               
See for example \cite[Lemma (2.5)]{Hi-Hu81}\omycite{Hi-Hu81}.

By \omyeqn{g4q} and \omyeqn{idb}, we have 
\begin{align}
g_4(q)=&\frac{J_{1,10}^3J_{3,10}J_{4,10}^3J_{5,10}^2J_{2,5}^4J_5^5}{qJ_{2,10}J_{1,5}^4J_{10}^6J_1^6}
-\frac{qJ_{1,10}J_{2,10}^3J_{3,10}^3J_{5,10}^2J_{1,5}^4J_5^5}{J_{4,10}J_{2,5}^4J_{10}^6J_1^6}
+4\frac{J_{1,10}^2J_{2,10}^2J_{3,10}^3J_{5,10}J_{2,5}^3J_5^5}{J_{1,5}^3J_{10}^6J_1^6}
\omylabel{eq:g4id}\\
&+4\frac{qJ_{1,10}^3J_{3,10}^2J_{4,10}^2J_{5,10}J_{1,5}^3J_5^5}{J_{2,5}^3J_{10}^6J_1^6}
-8\frac{J_{1,10}^3J_{3,10}^2J_{4,10}^2J_{5,10}J_{2,5}^2J_5^5}{J_{1,5}^2J_{10}^6J_1^6}
+8\frac{qJ_{1,10}^2J_{2,10}^2J_{3,10}^3J_{5,10}J_{1,5}^2J_5^5}{J_{2,5}^2J_{10}^6J_1^6}\nonumber\\
&-3\frac{J_{1,10}J_{2,10}^3J_{3,10}^3J_{5,10}^2J_{2,5}J_5^5}{J_{4,10}J_{1,5}J_{10}^6J_1^6}
-3\frac{J_{1,10}^3J_{3,10}J_{4,10}^3J_{5,10}^2J_{1,5}J_5^5}{J_{2,10}J_{2,5}J_{10}^6J_1^6}.\nonumber
\end{align}
We prove 
\beq
g_4(q)
=-4\frac{J_1^2J_4^3J_5J_{20}}{J_2^5J_{10}}+\frac{1}{q}\frac{J_2^4J_{10}^2}{J_1J_4^3J_{20}},
\omylabel{eq:g4id2}
\eeq
using the algorithm described in Section \subsect{geneta}.%
We first use \omyeqn{g4id} to rewrite \omyeqn{g4id2} as the following
modular function identity
for generalized eta-products on $\Gamma_1(20)$. 
\begin{align}
0=&1-{\frac{\eta_{{10,1}}^{6}\eta_{{10,4}}^{4}}{\eta_{{10,2}}^{4}\eta_{{10,3}}^{6}}}
+4\,{\frac{\eta_{{10,2}}^{2}\eta_{{10,3}}}{\eta_{{10,4}}^{2}\eta_{{10,5}}}}
+4\,{\frac{\eta_{{10,1}}^{7}\eta_{{10,4}}^{6}}{\eta_{{10,2}}^{6}\eta_{{10,3}}^{6}\eta_{{10,5}}}}
-8\,{\frac{\eta_{{10,1}}^{2}\eta_{{10,4}}}{\eta_{{10,2}}\eta_{{10,3}}\eta_{{10,5}}}}
+8\,{\frac{\eta_{{10,1}} ^{5}\eta_{{10,4}}^{3}}{\eta_{{10,2}}^{3}\eta_{{10,3}}^{4}\eta_{{10,5}}}}
\omylabel{eq:bigthetaid}\\
&-3\,{\frac{\eta_{{10,1}}\eta_{{10,2}}}{\eta_{{10,3}}\eta_{{10,4}}}}-3\,{\frac{\eta_{{5,1}}^{5}}{\eta_{{5,2}}^{5}}}
+4\,{\frac{\eta_{{20,1}}^{9}\eta_{{20,3}}^{3}\eta_{{20,4}}^{7}\eta_{{20,6}}^{4}\eta_{{20,7}}^{3}\eta_{{20,8}}^{3}\eta_{{20,9}}^{9}}{\eta_{{20,10}}^{2}}}
-{\frac{\eta_{{20,1}}^{6}\eta_{{20,2}}^{6}\eta_{{20,4}}^{7}\eta_{{20,6}}^{10}\eta_{{20,8}}^{3}\eta_{{20,9}}^{6}\eta_{{20,10}}^{2}}{\eta_{{20,5}}^{4}}}.
\nonumber
\end{align}
We use Theorem \thm{gepmodfunc} to check each that each generalized eta-product 
is
a modular function on $\Gamma_1(20)$. We then use Theorems \thm{chokoopark}
and \thm{cuspord1} to calculate the order of each generalized eta-product
at each cusp of $\Gamma_1(20)$. We calculate the constant in equation 
\omyeqn{Bdef1} to find that $B = 24$. We let $g(\tau)$ be the right side
of \omyeqn{bigthetaid} and easily show that $\ORD(g(\tau),\infty,\Gamma_1(20))> 24$.
The required identity follows by Theorem \thm{valcor1}.

From \omyeqn{f4g} and \omyeqn{g4id2} we have
\beq                  
\sum_{n=0}^\infty a_f(5n-1) q^n  + f(q^5) = q\,g_4(q)=
\frac{J_2^4J_{10}^2}{J_1J_4^3J_{20}}
-4q\,\frac{J_1^2J_4^3J_5J_{20}}{J_2^5J_{10}},
\omylabel{eq:f4gb}
\eeq
which is our result \omyeqn{af5id}.
\end{proof}
\subsection{A Fundamental Lemma}
\omylabel{subsec:fundlem5}

We need the following fundamental lemma,
whose proof follows easily from Theorem \thm{modeq}.
\begin{lemma}[A Fundamental Lemma]
\omylabel{lem:fun5}
Suppose $u=u(\tau)$, and $j$
is any integer. Then
\begin{align*}
{U_5}(u\,t^j) =-\sum_{l=0}^{4}\sigma_l(\tau)\,{U_5}(u\,t^{j+l-5}),
\end{align*}
where $t=t(\tau)$ is defined in \omyeqn{deft} and the $\sigma_j(\tau)$ are given 
in \omyeqn{sig0}--\omyeqn{sig4}.
\end{lemma}
\begin{proof}
The result follows easily from \omyeqn{modeq} by multiplying both
sides by $u\,t^{j-5}$ and applying $U_5$.
\end{proof}

\begin{lemma}\omylabel{lem:ordtpk} 
Let $u=u(\tau)$, and $l\in \mathbb{Z}$. 
Suppose for $l\leq k\leq l+4$
there exist Laurent polynomials $p_{u,k}(t)\in \mathbb{Z}[t,t^{-1}]$
such that 
\begin{align} U_5(u\,t^k)=v\,p_{u,k}(t),
\omylabel{eq:U5tk}
\end{align} 
and 
\begin{align} 
ord_t(p_{u,k}(t))\geq \CL{\frac{k+s}{5}}, 
\omylabel{eq:ordtpk} 
\end{align}
for a fixed integer $s$, where $t=t(\tau)$ is defined in
\omyeqn{deft} and where $v=v(\tau)$.
Then there exists a  sequence of Laurent
polynomials $p_{u,k}(t)\in \mathbb{Z}[t,t^{-1}]$, $k\in \mathbb{Z}$,
such that \omyeqn{U5tk} and \omyeqn{ordtpk} hold for all
$k\in \mathbb{Z}$.  
\end{lemma}

\begin{proof} 
We proceed by induction on $k$. Let $N > l + 4$ and assume the result
holds for $l \le k \le N-1$. Then by Lemma \omylem{fun5} we have
$$                  
U_5( u\,t^N) = - \sum_{j=0}^4 \sigma_j(\tau) \, U_5(u\,t^{N+j-5})
 = - v\,\sum_{j=0}^4 \sigma_j(\tau) \, p_{u,N+j-5}(t) = v\, p_{u,N}(t),
$$
where
$$
p_{u,N}(t) = 
  - \sum_{j=0}^4 \sigma_j(\tau) \, p_{u,N+j-5}(t) \in \mathbb{Z}[t,t^{-1}],
$$
and
\begin{align*}
\ord_t(p_{u,N}(t)) &\ge 
\umin{0 \le j \le 4}(\ord_t(\sigma_j) + \ord_t(p_{u,N+j-5}(t)))\\
&\ge\umin{0 \le j \le 4}\left(1 + \CL{\frac{N+j+s-5}{5}}\right)
 =\CL{\frac{N+s}{5}}.
\end{align*}
The result for all $k\ge l$ follows. The induction proof for $k < l$ is similar.
\end{proof}

\begin{lemma}
\omylabel{lem:U5utkcofs}
Let $u=u(\tau)$, and $l\in \mathbb{Z}$. 
Suppose for $l\leq k\leq l+4$
there exist Laurent polynomials $p_{u,k}(t)\in \mathbb{Z}[t,t^{-1}]$
such that 
\begin{align} U_5(u\,t^k)=v\,p_{u,k}(t),
\omylabel{eq:U5tkb}
\end{align} 
where
$$
p_{u,k}(t) = \sum_n c_u(k,n)\,t^n,\quad
\nu_5(c_u(k,n)) \ge \FL{\frac{3n-k+r}{4}}
$$
for a fixed integer $r$, where $t=t(\tau)$ is defined in
\omyeqn{deft} and where $v=v(\tau)$.
Then there exists a  sequence of Laurent
polynomials $p_{u,k}(t)\in \mathbb{Z}[t,t^{-1}]$, $k\in \mathbb{Z}$,
such that \omyeqn{U5tkb} holds for $k>l+4$, where
$$
p_{u,k}(t) = \sum_n c_u(k,n)\,t^n,\quad\mbox{and}\quad
\nu_5(c_u(k,n)) \ge \FL{\frac{3n-k+r+2}{4}}.
$$
\end{lemma}
\begin{remark}
Recall that $\nu_p(n)$ denotes the $p$-adic order of an integer $n$; i.e.\ the
highest power of $p$ that divides $n$.
\end{remark}

\begin{proof}
We proceed by induction on $k$. Let $N>l+4$ and assume \omyeqn{U5tkb}
holds for $l \le k \le N-1$ where
$$
p_{u,k}(t) = \sum_n c_u(k,n)\,t^n,\quad
\nu_5(c_u(k,n)) = \FL{\frac{3n-k+r}{4}}.
$$
As in the proof of Lemma \omylem{ordtpk} we have
$$                  
U_5( u\,t^N) 
= v\, p_{u,N}(t),
$$
where
$$
p_{u,N}(t) = 
  - \sum_{j=0}^4 \sigma_j(\tau) \, p_{u,N+j-5}(t) \in \mathbb{Z}[t,t^{-1}],
$$
From Lemma \omylem{fun5} we have
$$
\sigma_j(t) = \sum_{l=1}^{j+1} s(j,l)\,t^l \in \mathbb{Z}[t],
$$
where
$$
\nu_5(s(j,l)) \ge \FL{\frac{3l+j}{4}},
$$
for $1 \le l \le j+1$, $0 \le j \le 4$.
Therefore
\begin{equation*}
p_{u,N}(t) = - \sum_{j=0}^4\sum_{l=1}^{j+1} s(j,l) \sum_{m} c_u(N+j-5,m) t^{m+l}
= \sum_{n} c_u(N,n)\,t^n,
\end{equation*}
where
$$
c_u(N,n) = 
 - \sum_{j=0}^4\sum_{l=1}^{j+1} s(j,l) \, c_u(N+j-5,n-l),
$$
and
\begin{align*}
\nu_5(c_u(N,n)) & \ge 
\submin{1 \le l \le j+1 \\ 0 \le j \le 4} 
 \bigg(\nu_5(s(j,l))+\nu_5(c_u(N+j-5,n-l)\bigg)\\
& \ge
\submin{1 \le l \le j+1 \\ 0 \le j \le 4} 
 \bigg(\FL{\frac{3l+j}{4}} + \FL{\frac{3(n-l) - (N+j-5) + r}{4}}\bigg)\\
& \ge
\submin{1 \le l \le j+1 \\ 0 \le j \le 4} 
 \FL{\frac{3l+j+3(n-l) - (N+j-5) + r-3}{4}} \ge
 \FL{\frac{3n - N+ r +2}{4}}.
\end{align*}
The result follows.
\end{proof}

We define the following functions which will be needed in the proof of 
Theorem \omythm{mainthm}.
\begin{equation}
P_A:=\frac{J_{10}^2J_5J_2^6}{J_{20}J_4^3J_1^5}
      -4\frac{qJ_{20}J_5^2J_4^3}{J_{10}J_2^3J_1^2},\quad 
P_B:=\frac{J_{10}^6J_2^2J_1}{qJ_{20}^3J_5^5J_4}
      +4\frac{qJ_{20}^3J_4J_1^2}{J_{10}^3J_5^2J_2}\quad 
A:=\frac{J_{50}^2J_1^4}{J_{25}^4J_2^2},
\quad B:=\frac{qJ_{25}}{J_1}.
\omylabel{eq:PAPBABdefs}
\end{equation}
For $f=f(\tau)$ we define
\begin{equation}
U_A(f) := U_5(A\,f),\qquad U_B(f) := U_5(B\,f).
\omylabel{eq:UAUBdefs}
\end{equation}

First we need some initial values of $U_A(P_A\,t^k)$ and $U_B(P_B\,t^k)$.
\begin{lemma}
\omylabel{lem:initUAPAtk}
\begin{align*}
\mbox{Group I}& \\
&U_A(P_A)=P_B(5^4t^5-7\cdot5^3t^4+14\cdot5^2t^3-2\cdot5^2t^2+t),\\
&U_A(P_At^{-1})=-P_Bt,\\
&U_A(P_At^{-2})=-5P_Bt^2,\\
&U_A(P_At^{-3})=-5^2P_Bt^3,\\
&U_A(P_At^{-4})=-5^3P_Bt^4.\\
\mbox{Group II}& \\
&U_B(P_B)=P_A,\\
&U_B(P_Bt^{-1})=P_A(-5t+2),\\
&U_B(P_Bt^{-2})=P_A(5^2t^2-8\cdot5t+8),\\
&U_B(P_Bt^{-3})=P_A(5^3t^3-34\cdot5t+34),\\
&U_B(P_Bt^{-4})=P_A(-5^4t^4+16\cdot5^3t^3-36\cdot5^2t^2-128\cdot5t+6\cdot5^2).
\end{align*}
\end{lemma}
\begin{proof}
We use the algorithm described in Section \subsect{Upop} to prove
each of these identities. The identities take the form
$$
U_5(g) = f,
$$
where $f$, $g$ are linear combinations of eta-products.
For each identity we check that $f$ is a linear combination of eta-products
which are
modular functions on $\Gamma_0(100)$ and that $g$ is a linear
combination of eta-products
which are
modular functions on $\Gamma_0(20)$. For each of the identities we follow
the 5 steps in the algorithm given after Theorem \omythm{UpLB}. We note 
that the smallest value of $B$ encountered is $B=-14$. These steps have been
carried out with the help of \textsc{MAPLE}, including all necessary verifications
so that the results are proved.
\end{proof}

Following \cite{Pa-Ra12}\omycite{Pa-Ra12} a map $a\,:\,\mathbb{Z}\times\mathbb{Z} \longrightarrow
\mathbb{Z}$ is called a \textit{discrete array} if for each $i$ the
map $a(i,-)\,:\,\mathbb{Z} \longrightarrow \mathbb{Z}$, by $j \mapsto a(i,j)$
has finite support.

\begin{lemma}
\omylabel{lem:UAPAtk}
There exist discrete arrays $a$ and $b$ such that for $k \ge 1$
\begin{align}
U_A( P_A\,t^k ) &=  P_B\,\sum_{n\ge \CL{(k+5)/5}} a(k,n)\, t^n,\quad
\mbox{where}\quad \nu_5(a(k,n)) \ge \FL{\frac{3n-k}{4}},
\omylabel{eq:UAPAtk}\\
U_B( P_B\,t^k ) &=  P_A\,\sum_{n\ge \CL{k/5}} b(k,n)\, t^n,\quad
\mbox{where}\quad \nu_5(b(k,n)) \ge \FL{\frac{3n-k+2}{4}}.      
\omylabel{eq:UBPBtk}
\end{align}
\end{lemma}

\begin{proof}
From Lemma \omylem{initUAPAtk}, 
Group I we find there is a discrete array $a$ such that
$$
U_A( P_A\,t^k ) =  P_B\,\sum_{n\ge \CL{(k+5)/5}} a(k,n)\, t^n,\quad
\mbox{where}\quad \nu_5(a(k,n)) \ge \FL{\frac{3n-k-2}{4}},
$$
for $-4 \le k \le 0$.  Lemma \omylem{ordtpk} (with $s=4$) and Lemma 
\omylem{U5utkcofs} (with $r=-2$) imply \omyeqn{UAPAtk} for $k\ge1$.
From Lemma \omylem{initUAPAtk}, 
Group II we find there is a discrete array $b$ such that
$$
U_B( P_B\,t^k ) =  P_A\,\sum_{n\ge \CL{k/5}} b(k,n)\, t^n,\quad
\mbox{where}\quad \nu_5(b(k,n)) \ge \FL{\frac{3n-k}{4}},
$$
for $-4 \le k \le 0$.  Lemma \omylem{ordtpk} (with $s=0$) and Lemma 
\omylem{U5utkcofs} (with $r=0$) imply \omyeqn{UBPBtk} for $k\ge1$.
\end{proof}

\subsection{Proof of Theorem \omythm{mainthm}}
\omylabel{subsec:pfmainthmi}
For $\alpha\ge1$ define $\delta_\alpha$ by $0< \delta < 5^\alpha$
and $24\delta_\alpha\equiv1\pmod{5^\alpha}$. Then
$$
\delta_{2\alpha} =\frac{23\times5^{2\alpha} + 1}{24},\qquad 
\delta_{2\alpha+1} =\frac{19\times5^{2\alpha+1} + 1}{24}.
$$               
We let
$$
\lambda_{2\alpha}=\lambda_{2\alpha+1}= \frac{5}{24}(1 - 5^{2\alpha}).
$$
For $n\ge0$ we define
\begin{equation}
c_f(n) := a_f(5n-1)  + a_f(n/5).
\omylabel{eq:cfndef}
\end{equation}
We find that for $\alpha\ge3$
\begin{equation}
\sum_{n=0}^\infty \left(
a_f(5^{\alpha}n + \delta_\alpha)
+ a_f(5^{\alpha-2}n + \delta_{\alpha-2})\right)q^{n+1}
=\sum_{n=1}^\infty c_f(5^{\alpha-1}n + \lambda_{\alpha-1}) q^n.
\omylabel{eq:afcfid}
\end{equation}
We define the sequence of functions $(L_{\alpha})_{\alpha=0}^\infty$ by
$L_0 := P_A$ and for $\alpha\ge0$
$$
L_{2\alpha+1} := U_A( L_{2\alpha}),\qquad\mbox{and}\qquad
L_{2\alpha+2} := U_B( L_{2\alpha+1}).              
$$
\begin{lemma}
\omylabel{lem:Lalpha}
For $\alpha\ge0$,
$$
L_{2\alpha} = \frac{J_5 J_2^2}{J_1^4}
\sum_{n=0}^\infty c_f(5^{2\alpha}n + \lambda_{2\alpha}) q^n,
$$
and
$$
L_{2\alpha+1} = \frac{J_{10}^2 J_1}{J_5^4}
\sum_{n=0}^\infty c_f(5^{2\alpha+1}n + \lambda_{2\alpha+1}) q^n.
$$
\end{lemma}
\begin{proof}
\begin{align*}
L_0 &= P_A =
\frac{J_{10}^2J_5J_2^6}{J_{20}J_4^3J_1^5}
  -4q\frac{J_{20}J_5^2J_4^3}{J_{10}J_2^3J_1^2}
= \frac{J_5 J_2^2}{J_1^4} 
\left(
\frac{ J_{10}^2 J_2^4}{J_{20} J_4^3 J_1} - 4q
\frac{J_{20}J_5 J_4^3 J_1^2 }{ J_{10} J_2^5}
\right)\\
&= \frac{J_5 J_2^2}{J_1^4} 
\sum_{n=0}^\infty (a_f(5n-1) + a_f(n/5)) q^n =
\sum_{n=0}^\infty c_f(n+\lambda_0) q^n.
\end{align*}
This is the first equation with $\alpha=0$. The general result
follows by a routine induction argument.
\end{proof}

Our main result for the rank parity function modulo powers of $5$
is the following theorem.
\begin{theorem}
\omylabel{thm:Lmain}
There exists a discrete array $\ell$ such that for $\alpha\ge1$            
\begin{align}
L_{2\alpha} &=  P_A\,\sum_{n\ge 1} \ell(2\alpha,n)\, t^n,\quad
\mbox{where}\quad \nu_5(\ell(2\alpha,n)) \ge \alpha+ \FL{\frac{3n-3}{4}},
\omylabel{eq:L2a}\\
L_{2\alpha+1} &=  P_B\,\sum_{n\ge 2} \ell(2\alpha+1,n)\, t^n,\quad
\mbox{where}\quad \nu_5(\ell(2\alpha+1,n)) \ge \alpha+1+\FL{\frac{3n-6}{4}}.
\omylabel{eq:L2b}
\end{align}
\end{theorem}
\begin{proof}
We define the discrete array $\ell$ recursively. 
Define
\begin{align*}
&\ell(1,1)=1,\,
\ell(1,2)=-2\cdot 5^2,\,
\ell(1,3)=14\cdot 5^2,\,
\ell(1,4)=-7\cdot 5^3,\,
\ell(1,5)=5^4,\,\\ 
&\mbox{and}\quad \ell(1,k)=0, \quad\mbox{for $k\ge 6$}.
\end{align*}
For $\alpha\ge1$ define
\beq
\ell(2\alpha,n) = \sum_{k \ge 1} \ell(2\alpha-1,k)\,b(k,n) 
\qquad\mbox{(for $n\ge1$)},
\omylabel{eq:ellevdef}
\eeq
and
\beq
\ell(2\alpha+1,n) = \sum_{k \ge 1} \ell(2\alpha,k)\,a(k,n)
\qquad\mbox{(for $n\ge2$),}
\omylabel{eq:ellodddef}
\eeq
where $a$ and $b$ are the discrete arrays given in Lemma \lem{UAPAtk}.
From Lemma \omylem{initUAPAtk}, Group I 
and by Lemma \lem{UAPAtk} and equation \omyeqn{ellevdef} we have 
$$
L_1 = U_A(L_0) = U_A(P_A) = P_B \sum_{n=1}^5 \ell(1,n)\,t^n,\quad
\mbox{where}\quad \nu_5(\ell(1,n)) \ge \FL{\frac{3n-2}{4}}.
$$
\begin{align*}
L_2 &= U_B(L_1) = \sum_{n=1}^5 \ell(1,n) U_B(P_B t^n),\\
    &= \sum_{n=1}^5 \ell(1,n) P_A \sum_{k\ge1} b(n,k) t^k\\
    &= P_A\sum_{n\ge1} \sum_{k=1}^5\ell(1,k) b(k,n) t^n\\
    &= P_A \sum_{n\ge1} \ell(2,n) t^n,
\end{align*}
where
\begin{align*}
\nu_5(\ell(2,n))&\ge
\umin{1 \le k \le 5}
 \bigg(\nu_5(\ell(1,k))+\nu_5(b(k,n)\bigg)
\ge 
\umin{1 \le k \le 5}
 \bigg(\FL{\frac{3k-2}{4}} + \FL{\frac{3n-k+2}{4}}\bigg)
&=\FL{\frac{3n+1}{4}},
\end{align*}
since when $k=1$, 
 $\FL{\frac{3k-2}{4}} + \FL{\frac{3n-k+2}{4}}=\FL{\frac{3n+1}{4}}$, and
for $k\ge2$,
$$
\FL{\frac{3k-2}{4}} + \FL{\frac{3n-k+2}{4}}
\ge \FL{\frac{3n+2k-3}{4}}\ge \FL{\frac{3n+1}{4}}.
$$
Thus the result holds for $L_{2\alpha}$ when $\alpha=1$.
We proceed by induction. Suppose the result holds for $L_{2\alpha}$
for a given $\alpha\ge1$. Then by Lemma \lem{UAPAtk} and equation \omyeqn{ellodddef}
we have
\begin{align*}
L_{2\alpha+1} &= U_A(L_{2\alpha}) 
= \sum_{n\ge 1} \ell(2\alpha,n) U_A(P_A t^n),\\
    &= \sum_{n\ge1} \ell(2\alpha,n) P_B \sum_{k\ge2} a(n,k) t^k\\
    &= P_B\sum_{n\ge2} \sum_{k\ge1}\ell(2\alpha,k) a(k,n) t^n\\
    &= P_B \sum_{n\ge2} \ell(2\alpha+1,n) t^n,
\end{align*}
where
\begin{align*}
\nu_5(\ell(2\alpha+1,n))&\ge
\umin{1 \le k}
 \bigg(\nu_5(\ell(2\alpha,k))+\nu_5(a(k,n)\bigg)
\ge 
\umin{1 \le k}
 \bigg(\alpha + \FL{\frac{3k-3}{4}} + \FL{\frac{3n-k}{4}}\bigg)\\
&\ge\alpha + 1 + \FL{\frac{3n-6}{4}},
\end{align*}
since when $k=1$, 
 $\FL{\frac{3k-3}{4}} + \FL{\frac{3n-k}{4}}=1+\FL{\frac{3n-5}{4}}$, and
for $k\ge2$,
$$
\FL{\frac{3k-3}{4}} + \FL{\frac{3n-k}{4}}
\ge \FL{\frac{3n+2k-6}{4}}\ge 1+\FL{\frac{3n-6}{4}}.
$$
Thus the result holds for $L_{2\alpha+1}$.                    
Suppose the result holds for $L_{2\alpha+1}$
for a given $\alpha\ge1$. Then again by Lemma \lem{UAPAtk} 
and equation \omyeqn{ellevdef} we have
\begin{align*}
L_{2\alpha+2} &= U_B(L_{2\alpha+1}) 
= \sum_{n\ge 2} \ell(2\alpha+1,n) U_B(P_B t^n),\\
    &= \sum_{n\ge2} \ell(2\alpha+1,n) P_A \sum_{k\ge1} b(n,k) t^k\\
    &= P_A\sum_{n\ge1} \sum_{k\ge2}\ell(2\alpha+1,k) b(k,n) t^n\\
    &= P_A \sum_{n\ge1} \ell(2\alpha+2,n) t^n,
\end{align*}
where $\ell(2\alpha+1,1)=0$. Here 
\begin{align*}
\nu_5(\ell(2\alpha+2,n))&\ge
\umin{2 \le k}
 \bigg(\nu_5(\ell(2\alpha+1,k))+\nu_5(b(k,n)\bigg)
\ge 
\umin{2 \le k}
 \bigg(\alpha + 1+ \FL{\frac{3k-6}{4}} + \FL{\frac{3n-k+2}{4}}\bigg)\\
&\ge
\umin{2 \le k}
 \bigg( \alpha + 1 + \FL{\frac{3n+2k-7}{4}} \bigg)
=\alpha + 1 + \FL{\frac{3n-3}{4}}.  
\end{align*}
Thus the result holds for $L_{2\alpha+2}$, and the result holds in
general by induction.
\end{proof}

\begin{cor}
\omylabel{cor:cf5congs}
For $\alpha\ge1$ and all $n\ge0$ we have
\begin{align}
c_f(5^{2\alpha}n+\lambda_{2\alpha}) &\equiv 0 \pmod{5^\alpha},
\omylabel{eq:cf5acong}\\
c_f(5^{2\alpha+1}n+\lambda_{2\alpha+1}) &\equiv 0 \pmod{5^{\alpha+1}}.
\omylabel{eq:cf5bcong}
\end{align}
\end{cor}
\begin{proof}
The congruences follow immediately from Lemma \omylem{Lalpha}
and Theorem \omythm{Lmain}.
\end{proof}
In view of \eqn{afcfid} and Corollary \corol{cf5congs} we obtain
\omyeqn{rmod5}. This completes the proof of Theorem \omythm{mainthm}.

\section{Further results}                                  
\omylabel{sec:further}
The methods of this paper can be extended to study congruences mod powers of $7$ for both
the rank and crank parity functions. We describe some of these results, which we will prove
in a subsequent paper \cite{Ch-Ch-Ga20}\omycite{Ch-Ch-Ga20}. Analogous to \omyeqn{af5id}
we find that
\beq
\sum_{n=0}^{\infty}(a_f(n/7)-a_f(7n-2))q^n
=\frac{J_7^3}{J_2^2}\left(\frac{J_1^3J_7^3}{J_2^3J_{14}^3}+6q^2\frac{J_{14}^4J_1^4}{J_2^4J_7^4}\right),
\omylabel{eq:af7id}
\eeq
which leads to the following
\begin{theorem}
\omylabel{thm:rankthm7}
For all $\alpha\ge3$ and all $n\ge 0$ we have
\beq
a_f(7^{\alpha}n + \delta_\alpha)
- a_f(7^{\alpha-2}n + \delta_{\alpha-2})
\equiv 0 \pmod{7^{ \FL{\tfrac{1}{2}(\alpha-1)}}},
\omylabel{eq:rmod7}
\eeq
  where $\delta_\alpha$ satisfies $0 < \delta_\alpha < 7^\alpha$ and
$24\delta_\alpha\equiv1\pmod{7^\alpha}$.
\end{theorem}

It turns out that for the crank parity function congruences mod powers of $7$ 
are more difficult.
Define the \textit{crank parity function}
\beq
\beta(n) = M_e(n) - M_o(n),
\omylabel{eq:betadef}
\eeq
for all $n\ge0$.  The following is our analog of Choi, Kang and Lovejoy's Theorem \omythm{crankthm}.
\begin{theorem}
\omylabel{thm:crankthm7}
For each $\newaa\ge1$ there is an integral constant $K_\newaa$ such that
\beq
\beta(49n - 2) \equiv K_\newaa\, \beta(n) \pmod{7^{\newaa}},\qquad
\mbox{if $24n\equiv 1 \pmod{7^{\newaa}}$}.
\omylabel{eq:betamod7}
\eeq
\end{theorem}
This gives a weak refinement of Ramanujan's partition congruence 
modulo powers of $7$:
$$
p(n) \equiv 0 \pmod{7^\newaa},\qquad
\mbox{if $24n\equiv 1 \pmod{7^{\FL{\tfrac{a+2}{2}}}}$}.
$$
This was also proved by Watson \cite{Wa38}\omycite{Wa38}. Atkin and O'Brien 
\cite{At-OB}\omycite{At-OB} obtained congruences mod powers of $13$ for the 
partition function similar to \omyeqn{betamod7}.






\end{document}